\documentclass[12pt]{amsart}
\usepackage{amssymb}
\usepackage{amscd}
\usepackage{mathrsfs}
\usepackage{verbatim}
\usepackage{a4wide}

\theoremstyle{plain}
\newtheorem{theorem}{Theorem}[section]
\theoremstyle{remark}
\newtheorem{remark}[theorem]{Remark}

\theoremstyle{plain}
\newtheorem{corollary}[theorem]{Corollary}
\newtheorem{lemma}[theorem]{Lemma}
\newtheorem{proposition}[theorem]{Proposition}
\newtheorem{definition}[theorem]{Definition}

\numberwithin{equation}{section}
 \font\Bbbten=msbm10 \font\Bbbseven=msbm7 \font\Bbbfive=msbm5
\newfam\Bbbfam

  \textfont\Bbbfam=\Bbbten
  \scriptfont\Bbbfam=\Bbbseven
  \scriptscriptfont\Bbbfam=\Bbbfive
  \newcommand{\rA}{{\rm A}}
\newcommand{\rH}{{\rm H}}
\newcommand{\rK}{{\rm K}}
\def\R{{\mathbb R}\,}

\def\E{{\mathbb E}\,}
\def\P{\mathbb P}

\def\la{\left(}
\def\ra{\right)}
\def\lb{\langle}
\def\rb{\rangle}

\def\divv{{\rm div}\,}
\newcommand{\rpi}{{\rm \pi}}

\newcommand{\del}{}

\newcommand{\vp}{\varphi}
\newcommand\linspan{{\rm linspan}}

\newcommand{\embed}{\hookrightarrow}

\begin{document}

\baselineskip 13pt
\title[]
{Weak Solutions of the Stochastic Landau-Lifshitz-Gilbert Equation}
\author{Z. Brze{\'z}niak}
\address{Department of Mathematics\\
The University of York\\
Heslington, York YO10 5DD, UK} \email{zb500@york.ac.uk}
\author{B. Goldys}
\address{School of Mathematics\\
The University of New South Wales\\
Sydney 2052, Australia} \email{B.Goldys@unsw.edu.au}

\thanks{This work was partially supported by EPSRC Grant EP/E01822X/1 and ARC Discovery Project DP0663153}



\begin{abstract}
The Landau-Lifshitz-Gilbert equation perturbed by a multiplicative space-dependent noise is considered for a ferromagnet filling a bounded three-dimensional domain. We show the existence of weak martingale solutions taking values in a  sphere $\mathbb S^2$. The regularity of weak solutions is also discussed. Some of the regularity results are new even for the deterministic Landau-Lifshitz-Gilbert equation.
\end{abstract}

\maketitle \tableofcontents
\section{Introduction}
The study of the theory of ferromagnetism was initiated by Weiss (see \cite{brown_book} and references therein) and further developed by Landau and Lifshitz \cite{ll} and Gilbert \cite{gilbert}. According to this theory the orientation of the magnetic moment $u$ of a ferromagnetic material occupying  a region $D\subset\mathbb R^3$ at temperatures below the critical (so-called Curie) temperature satisfies, for $t>0$ and $x\in D$, the following Landau-Lifshitz-Gilbert (LLG) equation:
 \begin{equation}\label{e0}
 \frac{\partial u}{\partial t}(t,x)=\lambda_1u(t,x)\times H(t,x)-\lambda_2u(t,x)\times\left(u(t,x)\times H(t,x)\right),
 \end{equation}
 where $H$ is the so-called effective field. In the simplest situation, when the energy functional consists of the so-called exchange energy only, we have $H=\Delta u$ and we obtain the following version of the LLG equation:
\begin{equation}\label{e-1}
\left\{\begin{array}{ll}
  \frac{\partial u}{\partial t}=\lambda_1 u\times \Delta u- \lambda_2  u\times\left(u\times \Delta u\right),\;&t>0,\, x\in D,\\
 \frac{\partial u}{\partial n}(t,x)=0&t>0,\, x\in\partial D,\\
 u(0,x)=u_0(x)&x\in D,
 \end{array}\right.
\end{equation}
where $\lambda_1\in\R$, $\lambda_2>0$,  $n$ is the outer unit normal vector at the boundary $\partial D$ and we assume that at time $t=0$ the material is saturated, that is
\[|u_0(x)|=1.\]
Let us recall that the stationary solutions of this equation correspond to the equilibrium states of the ferromagnet and are not unique in general. An important problem in the theory of ferromagnetism is to describe phase transitions between different equilibrium states induced by thermal fluctuations of the field $H$. Therefore, the LLG equation needs to be modified in order to incorporate random fluctuations of the field $H$ into the dynamics of the magnetization $u$ and to describe noise-induced transitions between equilibrium states of the ferromagnet. The program to analyze noise induced transitions was initiated by N\'eel \cite{neel} and further developed in \cite{brown}, \cite{kamppeter} and others.
A simple way to incorporate the noise into the LLG equation is to perturb the field by white noise, that is to replace $H$ in \eqref{e0} by $H+h\dot{W}_t$, where $h:D\to\mathbb R^3$ is a given bounded function and $W$ is a one-dimensional Brownian Motion. Then we obtain a stochastic version of LLG equation \eqref{e-1}:
\begin{equation}\label{eqn_0}
\frac{\partial u}{\partial t}=\lambda_1 u\times(\Delta u+hdW)-\lambda_2 u\times (u\times(\Delta u+hdW)),
\end{equation}
but the question of how to understand the stochastic term in this equation is not obvious. There is strong evidence, , see \cite{palacios,kamppeter}, that the stochastic terms in \eqref{eqn_0} should be understood in the Stratonovitch sense and we adopt this approach throughout the paper. Moreover, following the arguments in  \cite{Kohn_R_VdE_2005} and \cite{palacios} we will assume that $\lambda_2$ is small in which case the noise in the second term above can be neglected. Finally, the stochastic version of the LLG equation that we are going to consider in this paper takes the form
\begin{equation}\label{eqn_SLL_1}
\left\{\begin{array}{rl}
d u (t)&=\left[\lambda_1 u (t)\times \Delta u(t)-\lambda_2 u(t) \times (u(t) \times \Delta u(t))\right]\, dt +\left(\lambda_3 u(t)\times h\right)\circ\, dW(t),\\
 \frac{\partial u}{\partial n}&=0, \mbox{ on } (0,\infty)\times \partial D.\\
 u(0,\cdot)&=u_0(\cdot).
 \end{array}\right.,
\end{equation}
where $\circ\, dW(t)$ stands for the Stratonovich differential. We emphasize that we consider one-dimensional noise for simplicity of presentation only and the case of $d$-dimensional white noise with $d>1$ requires only minor modifications. Let us note that this equation is closely related to other important partial differential equations and its mathematical analysis shares with them some difficulties. If $\lambda_1=\lambda_3=0$, $\lambda_2=1$ and equation \eqref{eqn_SLL_1} has smooth solutions $u$ such that $|u(t,x)|=1$ then equation \eqref{eqn_SLL_1} reduces to the equation for harmonic maps:
\[\frac{\partial u}{\partial t}=\Delta u+|\nabla u|^2u .\]
If $\lambda_2=\lambda_3=0$ then, using the Hashimoto transform, \eqref{eqn_SLL_1} can be transformed into the nonlinear Schr\"odinger equation, see \cite{faddeev}
\par\medskip\noindent
To the best of our knowledge equation \eqref{eqn_SLL_1} has until now been studied in the thin film approximation only, where it reduces to an ordinary stochastic differential equation, see the series of papers by R. V. Kohn and collaborators \cite{Kohn_R_VdE_2005}, \cite{Kohn_O_R_VdE_2007}. In these works the question of how to extend their results to the case of non-uniform (dependent on location) magnetisation is formulated as an open  problem. It should also be noted that in current applications of ferromagnetic materials such as magnetic memory elements the horizontal dimensions
of the device are comparable to the film thickness thus making such devices effectively three-dimensional, see for example \cite{abraham,fontana,scott}.  This provides the motivation to study equation \eqref{eqn_SLL_1} in its full three-dimensional version. In this case, \eqref{eqn_SLL_1} is a stochastic PDE closely related to the equations for harmonic maps with values in the sphere $\mathbb S^2$ and the global existence and regularity of solutions is far from obvious. In this paper we prove the existence of appropriately defined weak solutions. The existence and uniqueness of smooth solutions and the analysis of phase transitions will be the subject of the forthcoming paper.
\par\medskip\noindent
The paper is constructed as follows. In Section 2 equation \eqref{eqn_SLL_1} is reformulated as an evolution equation in the space $L^2\la D,\mathbb R^3\ra$ and  the notion of a weak martingale solution is made precise. Section 1 contains also the main result of this paper formulated as Theorem \ref{main}. Sections 3 and 4 are devoted to the proof of Theorem \ref{main}. In Section 3 we introduce the Faedo-Galerkin approximations and prove for them some uniform bounds in various norms. In Section 4 we use the Compactness method to show the existence of martingale solutions. Finally, in the Appendices we collected for the reader's convenience some facts scattered in the literature that are used in the course of the proof.
\par\medskip\noindent
The domain $D$ being fixed throughout the paper is omitted in the notation of relevant functional spaces. We will use the notation $\mathbb L^p$ for the space $L^p\left(D,\mathbb R^3\right)$, $\mathbb H^{1,2}$ for the Sobolev space $H^{1,2}\left(D,\mathbb R^3\right)$ and so on. We will denote by $\rH$ the Hilbert space $\mathbb L^2$. Occasionally we use the same notation for the spaces $L^2\left(D,\mathbb R^3\otimes\mathbb R^3\right)$, $H^{1,2}\left(D,\mathbb R^3\otimes\mathbb R^3\right)$ and so on.\\
We use the same notation $\lb\cdot ,\cdot\rb$ for the inner product in $\mathbb R^3$, $\mathbb R^3\otimes\mathbb R^3$ and in $\mathbb L^2$. It should not lead to confusion.
\section{Definition of a solution and the main results}
Let us observe that if we denote $G:\rH\ni u\mapsto u\times h\in \rH$, which is well defined provided $h\in \mathbb{L}^\infty$, then
$$
(Gu)\circ dW(t)=\frac12 G^\prime (u)[Gu]+ G(u) dW(t), \; u\in\rH.
$$
Since $G$ is a linear map, we infer that $G^\prime (u)[G(u)]=G^2u=(u\times h) \times h$. Thus we have
\[(Gu)\circ dW(t)= Gu \, dW(t)+\frac12(u\times h) \times h\,dt,\quad u\in\rH.\]
Denote by $\rA$ the $-$Laplacian with the Neumann boundary conditions acting on vector valued function, i.e.
\[
\left\{\begin{array}{ll}
D(\rA)&:=\{u\in \mathbb{H}^{2,2}: \frac{\partial u}{\partial n}=0 \;\text{ on }\; \partial D\},\\
\rA u&:=-\Delta u,\quad u\in \mathbb H^{2,2},
 \end{array}\right.\]
where $n=(n_1,n_2,n_3)$ is the unit outward normal vector field on $\partial D$ and $\frac{\partial u}{\partial n}:= \sum_{i} \frac{\partial u}{\partial x_i}n_i$.

It is well known that $\rm{A}$ is a positive self-adjoint operator in $\rH$ and that $(I+\rm{A})^{-1}$ is compact. Hence there exists an orthonormal basis $\{e_n\}_{n=1}^\infty$ of $\rH$ consisting of eigenvectors of $\rA$. Let us denote the corresponding eigenvalues by $(\mu_n)_{n=1}^\infty$. It is also known that
\[V:=D(\rA^{1/2})=\mathbb H^{1,2}.\]
Suppose that $u\in D(A)$ and $v\in \mathbb{H}^{1,2}$. Then by the Stokes Theorem we obtain
\begin{equation}\label{formula_Stokes}
\lb Au, v\rb_{\mathbb L^2}= \int_D \lb \nabla u(x),\nabla v(x)\rb\, dx-
\int_{\partial D} \left\langle \frac{\partial u}{\partial n}(y),v(y)\right\rangle\,
d\sigma(y),
\end{equation}
where $\sigma$ is the surface measure on $\partial D$.
If $u,v,z\in \mathbb{H}^{1,2}$ then
\begin{eqnarray}\label{formula_parts_1}
 \int_D \lb \nabla u(x), v(x)\rb\, dx-  \int_D \lb  u(x), \nabla v(x)\rb\, dx& = &  \sum_{j=1}^3 \int_{\partial D} \lb  u (y),v(y)\rb\,\nu_j(y)  d\sigma(y),\\
 \label{formula_parts_2}
 \int_D \lb \nabla u(x), z(x)\rb\, dx-  \int_D \lb  u(x), \divv z(x)\rb\, dx& = &   \int_{\partial D}   u (y)\, \lb z(y),\nu(y)\rb\,  d\sigma(y).
 \end{eqnarray}
The definition of weak solution will be preceded by some identities, mostly following the Visintin's paper \cite{Vis_85}.
\begin{proposition}\label{prop_Green}
 If $v \in V$ and $u\in D(A)$, then
\begin{equation}\label{eqn_Green}
\int_D \lb u(x)\times Au(x),v(x)\rb=\sum_{i}\int_D \left\langle \frac{\partial u}{\partial x_i}(x), \frac{\partial(v\times  u)}{\partial x_i}(x)\right\rangle \, dx\, .
\end{equation}
\end{proposition}
\begin{proof}[Proof of formula (\ref{eqn_Green})] Because $u\in \mathbb H^{2,2}$,  $\rA u\in\mathbb L^2$ and  by the Sobolev-Gagliardo inequalities $u\in \mathbb L^\infty$. Hence, $u\times Au$ belongs to $\mathbb L^2$ and the LHS of  (\ref{eqn_Green}) is well defined. Invoking again the Sobolev-Gagliardo inequalities we obtain $v, D_iv\in \mathbb L^6$, $D_iv\in\mathbb L^2$,  $i=1,2,3$,   and
$u\in  \mathbb L^\infty$ and $D_iu\in \mathbb L^2$, $j=1,2,3$. Therefore the RHS of equality (\ref{eqn_Green}) is well defined as well.
Notice next that because of formula (\ref{eqn_A-2}) we have
$$\int_D \lb u(x)\times Au(x),v(x)\rb= \int_D \left\langle  Au(x),v(x) \times u(x)\right\rangle=-\sum_{i}  \int_D \left\langle \frac{\partial^2 u}{\partial x_i^2}(x), v(x) \times u(x)\right\rangle\, dx.$$
Applying the Stokes Theorem to the last integral, see e.g. \cite[Theorem 1.2, p.7]{Temam_2001}, we infer that
\begin{eqnarray*}
&&\hspace{-7truecm}\lefteqn{\int_D \lb u(x)\times Au(x),v(x)\rb
=  \int_D \sum_{i}  \left\langle \frac{\partial u}{\partial x_i}(x),\frac{\partial(v(x)\times  u(x))}{\partial x_i}\right\rangle \, dx}
\\
&-&  \int_{\partial D} \sum_{i}  \nu_i \left\langle \frac{\partial u}{\partial x_i}(x), v(x)\times  u(x)\right\rangle\, d\sigma_x  \\
=  \int_D \sum_{i} \left\langle \frac{\partial u}{\partial x_i}(x),\frac{\partial(v\times  u)}{\partial x_i}(x)\right\rangle \, dx
&-&  \int_{\partial D}  \left\langle \frac{\partial u}{\partial \nu}(x), v(x)\times  u(x)\right\rangle\, d\sigma(x),
\end{eqnarray*}
and \eqref{eqn_Green} follows.
\end{proof}
\begin{lemma}\label{lem_sym}
 If $v \in V$ and $u\in D(\rA)$, then
 \begin{equation}
\label{eqn_sym}
\sum_{i}\int_D \left\langle \frac{\partial u}{\partial x_i}(x), v(x)\times \frac{\partial  u}{\partial x_i}(x)\right\rangle \, dx =0.
\end{equation}
\end{lemma}
\begin{proof}
Taking into account that $\lb a\times b,b\rb=0$ it is enough to show that the integral in \eqref{lem_sym} is well defined but this follows immediately from the Gagliardo-Nirenberg inequality
\end{proof}
Thus, we get the following result as a direct consequence of Proposition \ref{prop_Green} and Lemma \ref{lem_sym}.
\begin{corollary}\label{cor_Green}
 If $v \in V$ and $u\in D(A)$, then
\begin{equation}
\label{eqn_Green2}
\int_D \lb u(x)\times Au(x),v(x)\rb dx=\sum_{i}\int_D \left\langle \frac{\partial u}{\partial x_i}(x), \frac{\partial v}{\partial x_i}(x) \times  u(x) \right\rangle \, dx,
\end{equation}
\end{corollary}
Now we are ready to formulate the definition of a solution.
\begin{definition}\label{def_sol_mart}
A weak martingale solution  $\left(\Omega ,{\mathcal F}, ({\mathcal F}_t)_{t\geq 0}, \mathbb{P},W,u\right)$ to equation (\ref{eqn_SLL_1}) consists of a filtered  probability space $(\Omega ,{\mathcal F}, ({\mathcal F}_t)_{t\geq 0}, \mathbb{P})$ with the filtration satisfying the usual conditions,
a one dimensional  $\left(\mathcal F_t\right)$-adapted Wiener process  $W=(W_t)_{t\geq 0}$,   and a progressively measurable process $u:[0,\infty) \times\Omega\to \rH$ such that:\\
(a) for $\mathbb{P}$-a.s. $\omega\in\Omega$, for all $T>0$,
\[u(\cdot,\omega)\in C\left([0,T];\mathbb H^{-1,2}\right)\]
(b) For every $T>0$
\[\E\left(\sup_{t\le T}|u(t)|^2_{\mathbb L^2}+\sup_{t\le T}|\nabla u(t)|^2_{\mathbb L^2}\right)<\infty ,\]
(c) For every $\varphi\in C_0^\infty\left(D,\mathbb R^3\right)$ and every $t\ge 0$ the following equation is satisfied $\mathbb P$-a.s.:
\begin{eqnarray}\nonumber
\lb u(t),\vp\rb -  \lb u_0,\vp\rb &= & \lambda_1\int_0^t \sum_{i}\int_D \left\langle \frac{\partial u}{\partial x_i}(s,x), \frac{\partial \vp}{\partial x_i}(x) \times  u(s,x) \right\rangle \, dx \, ds \\
& & -\lambda_2\int_0^t\sum_i \int_D  \left\langle \frac{\partial u}{\partial x_i}(s,x), \frac{\partial (u\times\vp)}{\partial x_i}(s,x) \times  u(s,x) \right\rangle \, dx \, ds \\
\nonumber &&+  \lambda_3\int_0^t \int_D  \lb u(s,x)\times h(x),\vp(x)\rb  \,\circ \, dW(s) .
\label{eqn_snse_weak}
\end{eqnarray}
\end{definition}
In the next sections we will prove the existence of a weak solution which is in fact solutions to equation \eqref{eqn_SLL_1} in a stronger sense. In order to formulate our main result we need some preparations. If $v\in D(\rA)$ then $v\times\rA v\in\mathbb L^2$ hence defines an element of $\mathbb H^{-1,2}$. Consider $v\in\mathbb H^{1,2}$ and $Y\in\mathbb L^2$ such that for a certain sequence $\la v_n\ra\subset D(\rA)$ we have
\begin{equation}\label{conv}
\left|v_n-v\right|_{\mathbb H^1}+\left|v_n\times\Delta v_n-Y\right|_{\mathbb L^2}\longrightarrow 0.
\end{equation}
Then by Proposition \ref{prop_Green}
\begin{equation}\label{udelta}
\langle Y,\phi\rangle=\sum_i\int_D\left\langle \frac{\partial v}{\partial x_i}(x),\frac{\partial(\varphi\times v)}{\partial x_i}(x)\right\rangle dx,\quad \phi\in C_0^\infty\left(D,\mathbb R^3\right) .
\end{equation}
It follows that for $v\in\mathbb H^{1,2}$ we have $Y=v\times\Delta v$ in distributional sense. In a similar way, if for $v\in\mathbb H^{1,2}$ and  $Y\in\mathbb L^2$ \eqref{conv} holds and  $Z\in\mathbb L^{6/5}$ is such that
\[\left|v_n\times\left(v_n\times\Delta v_n\right)-Z\right|_{\mathbb L^{6/5}}\longrightarrow 0\]
then we identify $Z$ with $v\times(v\times\Delta v)\in\mathbb L^{6/5}$.
Now we can formulate the main result of this paper.
\begin{theorem}\label{main}
Assume that $u_0\in\mathbb H^{1,2}$, $h\in \mathbb H^{1,\infty}$ and $\left|u_0(x)\right|=1$ $Leb-a.e.$. Then there exists a weak martingale solution $\left(\Omega ,{\mathcal F}, ({\mathcal F}_t)_{t\geq 0}, \mathbb{P},W,u\right)$ to equation \eqref{eqn_SLL_1} such that:\\
(a) for every $T>0$
 \begin{equation}\label{eqn_def_SLL_2}
 \mathbb{E} \int_0^T \vert  u(t) \times \Delta u(t) \vert_{\mathbb{L}^2}^2\, dt <\infty .
 \end{equation}
(b) For every $t\ge 0$
\[u(t)=u_0+\lambda_1\int_0^tu(s)\times\Delta u(s)ds-\lambda_2\int_0^tu(s)\times(u(s)\times\Delta u(s))ds+\lambda_3\int_0^t(u(s)\times h)\circ dW(s),\quad\mathbb P-a.s.\]
where the first two integrals are the Bochner integrals in $\mathbb L^{2}$ and the Stratonovich integral is well defined in $\mathbb L^2$.\\
(c) For every $T>0$ and $\alpha\in\la 0,\frac{1}{2}\ra$
\[u(\cdot)\in C^{\alpha}\la [0,T],\mathbb L^{2}\ra ,\quad \mathbb P-a.s.\]
(d) For all $t\ge 0$
\begin{equation}\label{sphere}
|u(t,x)|=1,\quad\mathrm{for}\quad Leb-a.e.\,\, x.
\end{equation}
\end{theorem}
It seems that part (b) of the theorem is new even in the deterministic case when $\lambda_3=0$.\\
The remaining part of the paper is devoted to the proof of this theorem.
\section{Faedo-Galerkin approximation}
Let us denote by $\rpi_n$ the orthogonal projection from $\rH$ onto $\rH_n:=\linspan\{e_1,\cdots,e_n\}$. Let us note that because of Proposition \ref{prop-1} and the fact that $\rA(\rH_n)\subset \rH_n\subset L^\infty(D,\mathbb{R}^3)$  we have the following
\begin{lemma}\label{lem-1}
Let us consider the maps
\begin{eqnarray}
F^1_n &:&\rH_n\ni u\mapsto \rpi_n(u \times \Delta u)\in \rH_n\\
F^2_n &:&\rH_n\ni u\mapsto \rpi_n(u \times( u \times \Delta u))\in \rH_n\\
G_n &:&\rH_n\ni u\mapsto \rpi_n(u \times h)\in \rH_n
\end{eqnarray}
Then the maps $F_n^1$ and $F_n^2$ are Lipschitz on balls, the map $G_n$ is linear and
\begin{equation}\label{ineq_G_n}
\vert G_nu\vert_{\rH_n} \leq  \vert h\vert_{\mathbb{L}^\infty}\vert u\vert_{\rH}, \; u\in \rH_n.
\end{equation}
\end{lemma}
Let $F_n=\lambda_1F^1_n-\lambda_2F^2_n$.
We have the following simple consequences of identity $\lb a\times b,b\rb=0$ and the fact that $\rpi_n$ is self-adjoint.
\begin{lemma}\label{lem-2} Assume that $h\in \mathbb L^\infty$. Then for all $u\in \rH_n $, and $i=1,2$,
\[ \lb F^i_n(u), u\rb_{\rH} =0\quad\mathrm{and}\quad \lb G_nu, u\rb_{\rH} =0.\]
\end{lemma}
In view of the above two lemmata the following stochastic differential equation on $\rH_n$
\begin{equation}\label{eqn_n}
\left\{\begin{array}{rl}
d u_n (t)&= F_n(u_n(t))\, dt + G_nu(t)\circ \,dW(t),\\
 u_n(0)&=\rpi_nu_0
 \end{array}\right.
\end{equation}
has a unique strong global solution $u_n=\big(u_n(t)\big)$, $t\geq 0$.
Note that, putting
\[\hat{F}_n(u)=F_n(u)+\frac12 G_n^2u=F_n(u)+\frac12{\lambda_3^2} \rpi_n[(\rpi_n(u\times h))\times h],\]
 the stochastic differential equation in \eqref{eqn_n}  can be written in the following It\^o form
\begin{equation*}
d u_n (t)= \hat{F}_n(u_n(t))\, dt + G_nu_n(t) \,dW(t),
\end{equation*}
Our aim it to prove the following a priori estimates.
\begin{theorem}\label{thm_apriori} Assume that $h\in \mathbb H^{1,\infty}$. Then for every $t\ge 0$ and $n\ge 1$
\begin{equation}\label{ineq_ap_1}
\vert u_n(t)\vert_{\mathbb{L}^{2}} =\vert u_n(0)\vert_{\mathbb{L}^{2}},\quad \mathbb{P}-\mbox{a.s}
\end{equation}
Moreover, there exists a constant $C>0$ such that for all $n\in\mathbb{N}$,
\begin{eqnarray}
\label{ineq_ap_2}
\mathbb{E}\left[ \sup_{t\in [0,T]} \vert \nabla u_n(t)\vert_{\mathbb{L}^{2}}^2\right] &\leq &C,\\
\label{ineq_ap_3}
\mathbb{E} \int_0^T \vert  u_n(t) \times \Delta u_n(t) \vert_{\mathbb{L}^{2}}^2\, dt &\leq &C,\\
\label{ineq_ap_4}
\mathbb{E} \int_0^T \vert u_n(t) \times \big( u_n(t) \times \Delta u_n(t)\big) \vert_{\mathbb{L}^{6/5}}^{\frac43}\, dt &\leq &C.
\end{eqnarray}
\end{theorem}
\begin{proof}[Proof of \eqref{ineq_ap_1}]
We will apply the It\^o Lemma to a process $\varphi\left(u_n\right)$, where $\varphi:\rH_n\ni u\mapsto \frac12 \vert u\vert_{\rH}^2\in\mathbb{R}$. Since $G_n^\ast=-G_n$, in view  of Lemmata  \ref{lem-1} and \ref{lem-2} we obtain
\[\begin{aligned}
\frac{1}{2}d\left|u_n\right|^2&=\lb \hat{F}_n\la u_n\ra ,u_n\rb dt+\lb G_nu_n ,u_n\rb dW+\frac{1}{2}\left|G_nu_n\right|^2dt\\
&=\frac{1}{2}\lb G_n^2u_n,u_n\rb dt-\frac{1}{2}\lb G_n^2u_n,u_n\rb dt=0\, .
\end{aligned}\]
Therefore, as required
$$
\frac12 \vert u_n(t)\vert_{\rH}^2= \frac12 \vert u_n(0)\vert_{\rH}^2\del{+\frac12 \vert G_nu_n(t)\vert_{\rH}^2},\; t\geq 0\, .
$$
\end{proof}
\begin{proof}[Proof of inequality \eqref{ineq_ap_2}]
Let now
\[\phi(u)=\frac{1}{2}|\nabla u|_{\rH}^2,\quad u\in\rH_n .\]
Then
\[\lb\phi^\prime(u),g\rb=\lb\nabla u,\nabla g\rb=-\lb \Delta u,g\rb ,\quad u,g\in\rH_n\]
and
\[\lb\phi^{\prime\prime}(u)g,k\rb=\lb\nabla g,\nabla k\rb ,\quad u,g,k\in\rH_n .\]
Therefore, for all $t\ge 0$ the It\^o formula yields
\[
\begin{aligned}
\phi\la u_n(t)\ra-\phi\la u_n(0)\ra &=\int_0^t\lb \nabla u_n(s),\nabla\hat{F}_n\la u_n(s)\ra\rb ds\\
&+\int_0^t\lb\nabla u_n(s) ,\nabla G_nu_n(s)\rb dW(s)+\frac{1}{2}\int_0^t\left|\nabla G_nu_n(s)\right|^2_{\rH} ds\, .
\end{aligned}\]
Since, by Lemma \ref{lem-2}
\[\begin{aligned}
\lb \nabla u_n(s),\nabla\hat{F}_n\la u_n(s)\ra\rb&=\lb \nabla u_n(s),\nabla F_n\la u_n(s)\ra\rb+\frac{1}{2}\lb \nabla u_n(s),\nabla G_n^2u_n(s)\rb\\
&=-\lb \Delta u_n(s),F_n\la u_n(s)\ra\rb+\frac{1}{2}\lb \nabla u_n(s),\nabla G_n^2u_n(s)\rb\\
&=-\lambda_2\left|u_n(s)\times \Delta u_n(s)\right|^2_{\rH}+\frac{1}{2}\lb \nabla u_n(s),\nabla G_n^2u_n(s)\rb
\end{aligned} \]
we obtain for $t\ge 0$
\begin{equation}\label{ben1}
\begin{aligned}
\frac{1}{2}\left|\nabla u_n(t)\right|^2_{\rH}&+\lambda_2^2\int_0^t\left|u_n(s)\times \Delta u_n(s)\right|^2_{\rH}ds\\
&=\frac{1}{2}\left|\nabla u_n(0)\right|^2_{\rH}ds+\int_0^t\lb\nabla u_n(s) ,\nabla G_nu_n(s)\rb dW(s)\\
&+\frac{1}{2}\int_0^t\left|\nabla G_nu_n(s)\right|^2_{\rH} ds+\frac{1}{2}\int_0^t\lb \nabla u_n(s),\nabla G_n^2u_n(s)\rb ds\, .
\end{aligned}
\end{equation}
Next, using the definition of $\rA$
 we find that
\[\begin{aligned}\left|\nabla G_nu_n(s)\right|^2_{\rH}&=\lambda_3^2\left|\rA^{1/2}\la\pi_n\la u_n(s)\times h\ra\ra\right|^2_{\rH}\le \lambda_3^2\left|\rA^{1/2}\la u_n(s)\times h\ra\right|^2_{\rH}\\
&= \lambda_3^2\left|\nabla\la u_n(s)\times h\ra\right|^2_{\rH}=\lambda_3^2\left|\la\nabla u_n(s)\ra\times h+u_n(s)\times\nabla h\right|^2_{\rH}\\
&\le 2\lambda_3^2\la \left|\nabla u_n(s)\right|_{\rH}^2|h|_{L^\infty}^2+\|\nabla h|^2_{\rH}\ra=a\left|\nabla u_n(s)\right|^2_{\rH}+b
\end{aligned}\]
and similarly
\[\left|\nabla G^2_nu_n(s)\right|_{\rH}^2\le a_1\left|\nabla u_n(s)\right|^2_{\rH}+b_1\]
Therefore,
\[\sup_{r\le t}\left|\nabla u_n(r)\right|^2_{\rH}\le \left|\nabla u_n(0)\right|^2_{\rH}+\sup_{r\le t}\left|\int_0^r\lb\nabla u_n(s) ,\nabla G_nu_n(s)\rb dW(s)\right|+\int_0^t\left(a_1\left|\nabla u_n(s)\right|^2_{\rH}+b_1\right)ds\, .
\]
Hence invoking the Burkholder-Davis-Gundy inequality we find that for any $T>0$ and $t\le T$
\[\begin{aligned}
\E\sup_{r\le t}\left|\nabla u_n(r)\right|^2_{\rH}&\le \E\left(\int_0^t\lb\nabla u_n(s) ,\nabla G_nu_n(s)\rb^2ds\right)^{1/2}+\int_0^t\E\left(a_2\left|\nabla u_n(s)\right|^2+b_2\right)ds\\
&\le \left(\E\int_0^t\lb\nabla u_n(s) ,\nabla G_nu_n(s)\rb^2ds\right)^{1/2}+\int_0^t\E\left(a_2\left|\nabla u_n(s)\right|^2+b_2\right)ds\\
&\le a_T+b_T\int_0^t\E\sup_{r\le s}\left|\nabla u_n(r)\right|^2ds
\end{aligned}
\]
for some constants $a_T,b_T>0$ that are independent of $n$. Finally, using the Grownall inequality we obtain \eqref{ineq_ap_2}.
\end{proof}
\begin{proof}[Proof of inequality \eqref{ineq_ap_3}]
It follows immediately from \eqref{ben1} and \eqref{ineq_ap_2} since $\lambda_2>0$.
\end{proof}
\begin{proof}[Proof of inequality \eqref{ineq_ap_4}] Let us notice that the following inequality is a special case of inequality \eqref{ineq_L^6} with $r=\frac65$ (so that $\frac{3}{r}-2=\frac12$).
\begin{eqnarray}\label{ineq_L^6b}
\vert u\times v\vert_{_{L^\frac65}} &\leq&  \vert u \vert_{_{L^2}}^{\frac12} \vert u \vert_{_{L^6}}^{\frac12} \vert v\vert_{_{L^2}}.
\end{eqnarray}
Let us fix $q,p\in(1,\infty)$ and let $p^\ast$ be the conjugate exponent to $p$, i.e. $\frac1p+\frac1{p^\ast}=1$.
 Then, estimates (\ref{ineq_L^6b}) and (\ref{ineq_ap_1}) yield
\newcommand{\cosik}{q}
\begin{eqnarray*}
&& \mathbb{E} \int_0^T \vert u_n(t) \times \big( u_n(t) \times \Delta u_n(t)\big) \vert_{\mathbb{L}^{6/5}}^\cosik\, dt \\
&\leq & \left|u_n(0)\right|^{q/2}_{\mathbb L^2}\mathbb{E} \int_0^T \vert u_n(t)\vert_{_{L^2}}^{q/2} \vert u_n(t)\vert_{_{L^6}}^{q/2} \vert  u_n(t) \times \Delta u_n(t) \vert_{\mathbb{L}^{2}}^\cosik\, dt\\
&\leq & \left|u_n(0)\right|^{q/2}_{\mathbb L^2}\mathbb{E}\left( \sup_{t\in [0,T]} \vert u_n(t)\vert_{_{\mathbb{L}^6}}^{q/2}\int_0^T  \vert  u_n(t) \times \Delta u_n(t) \vert_{\mathbb{L}^{2}}^\cosik\, dt\right).
\end{eqnarray*}
Therefore, using (\ref{ineq_ap_2}), invoking the H\"older inequality and the Sobolev embedding theorem we obtain
\begin{eqnarray*}
&& \mathbb{E} \int_0^T \vert u_n(t) \times \big( u_n(t) \times \Delta u_n(t)\big) \vert_{\mathbb{L}^{6/5}}^\cosik\, dt \\
&\leq & \left|u_n(0)\right|^{q/2}_{\mathbb L^2}\left[ \mathbb{E} \sup_{t\in [0,T]} \vert u_n(t)\vert_{_{\mathbb{L}^6}}^{pq/2}\right]^\frac1p \left[  \mathbb{E} \big\vert  \int_0^T  \vert  u_n(t) \times \Delta u_n(t) \vert_{\mathbb{L}^{2}}^\cosik\, dt\big\vert^{p^\ast} \right]^{\frac1{p^\ast}}\\
&\leq & \left|u_n(0)\right|^{q/2}_{\mathbb L^2}\left[ \mathbb{E} \sup_{t\in [0,T]} \vert u_n(t)\vert_{_{\mathbb{L}^6}}^{pq/2}\right]^\frac1p \left[  \mathbb{E}  T^{\frac{p^\ast}{p}} \int_0^T  \vert  u_n(t) \times \Delta u_n(t) \vert_{\mathbb{L}^{2}}^{p^\ast\cosik}\, dt \right]^{\frac1{p^\ast}}
\end{eqnarray*}
Choosing $p$ and $q$ such that $\frac{pq}{2}=p^\ast\cosik=2$, i.e. $q=\frac43$, $p=3$ concludes the proof.
\end{proof}
Next, in order to deal with $H^{1,2}$ norm of the solutions, we formulate the following simple but fundamental properties:
\begin{lemma}\label{lem-3} For all $u\in \rH_n$, and $i=1,2$,
\begin{eqnarray}
\lb F^1_n(u), \Delta u\rb_{\rH} &=&0\\
\lb F^2_n(u), \Delta u\rb_{\rH} &=& -\vert u \times \Delta u\vert_{\rH}^2\\
\label{ineq-difficulty} \vert \lb \rpi_n[(\rpi_n(u\times h))\times
h] , \Delta u\rb_{\rH} \vert  &\leq & \vert  \lb \lb u,h\rb h-|
h|^2u , \Delta u  \rb_{\rH} \vert.
\end{eqnarray}
Moreover, there exist  constants $C_1,C_2$, depending on the vector field $h$, such that for all $u\in \rH_n$,
\begin{eqnarray}
\lb \hat{F}_n(u), \Delta u\rb_{\rH} &\leq &\lambda_2 \vert u \times
\Delta u\vert_{\rH}^2+C_1\vert u\vert_{\rH} \vert \nabla
u\vert_{\rH}+ C_2\vert \nabla u\vert_{\rH}^2.
\end{eqnarray}
\end{lemma}
\begin{proof} Let us fix $u\in \rH_n$. We begin with the  proof of the first three claims.  Because $\Delta u\in\rH_n$ for $u\in\rH_n$, we have  the following sequences of equalities.
\begin{eqnarray*}
\lb F^1_n(u), \Delta u\rb_{\rH} &=& \lb \rpi_n(u\times \Delta u), \Delta u\rb_{\rH} =\lb u\times \Delta u, \rpi_n(\Delta u)\rb_{\rH}\\
&=&\lb u\times \Delta u, \Delta u\rb_{\rH} =0.\\
\lb F^2_n(u), \Delta u\rb_{\rH} &=& \lb \rpi_n((u\times(u\times \Delta u)), \Delta u\rb_{\rH} =\lb u\times(u\times \Delta u), \rpi_n(\Delta u)\rb_{\rH}\\
\mbox{by \eqref{eqn_A-10}}\;\;\;\;\;\;\;\;\;\;\;\;\;\;\;\;\;\;\;\;\;\;\;\;
&=&\lb u\times(u\times \Delta u), \Delta u \rb_{\rH} = -\vert u \times \Delta u\vert_{\rH}^2.\\
\vert \lb \rpi_n[(\rpi_n(u\times h))\times h], \Delta u\rb_{\rH} \vert &=& \vert \lb (\rpi_n(u\times h))\times h, \rpi_n (\Delta u) \rb_{\rH} \vert\\
\vert  \lb \rpi_n(u\times h)\times h, \Delta u \rb_{\rH} \vert&=&  \vert - \lb \rpi_n(u\times h), \Delta u \times h \rb_{\rH} \vert\\
= \vert  \lb \rpi_n(u\times h), \rpi_n(\Delta u \times h )\rb_{\rH} \vert &\leq&  \vert  \lb u\times h, \Delta u \times h \rb_{\rH} \vert\\
=\vert  \lb (u\times h)\times h , \Delta u  \rb_{\mathbb L^2} \vert &=&\vert  \lb \lb u,h\rb h-| h|^2u , \Delta u  \rb_{\mathbb L^2} \vert\, ,
\end{eqnarray*}
where $| \cdot|$ and $\lb\cdot,\cdot\rb$ denote the  $\mathbb{R}^3$-norm and inner product respectively. This concludes the proof of inequality \eqref{ineq-difficulty}.

Next, employing the Stokes formula \eqref{formula_Stokes}, and noticing that $\frac{\partial u}{\partial n}(y)=$ for $y\in\partial D$,  we have the following
\begin{eqnarray*}
\lb \lb u,h\rb h , -\Delta u  \rb_{\rH}&=&  \int_D\lb \nabla [\lb u(x),h(x)\rb h(x)] , \nabla u(x)  \rb \, dx-\int_{\partial D}\left\langle  \frac{\partial u}{\partial n}(y), \lb u(y),h(y) \rb h(y) \right\rangle\, d\sigma(y)\\
&=&\sum_{ijk} \int_D \big(D_ju_i(x)h_i(x)+u_i(x)D_jh_i(x)\big)h_k(x)D_ju_k(x)\, dx\\
\mbox{ and } &&\\
\left\langle -| h|^2u  , \Delta u  \right\rangle_{\mathbb L^2} &=& \int_D\left\langle \nabla [| h(x)|^2u(x) ] , \nabla u(x)  \right\rangle \, dx-\int_{\partial D}\left\langle \frac{\partial u}{\partial n}(y),  | h(y)|^2u(y) \right\rangle\, d\sigma(y)\\
&=&\sum_{ijk} \int_D \big(2h_i(x)D_jh_i(x)u_k(x)+h_i^2(x)D_ju_k(x)\big)D_ju_k(x)\, dx\\
\end{eqnarray*}
Hence, there exist  constants $C_1,C_2$, depending on various norms of the vector field $h$, such that
\begin{eqnarray*}
\vert \left\langle \lb u,h\rb h -| h|^2u, \Delta u  \right\rangle_{\mathbb L^2}\vert &\leq& C_1\vert u\vert_{\rH} \vert \nabla u\vert_{\mathbb L^2}+ C_2\vert \nabla u\vert_{\mathbb L^2}^2.
\end{eqnarray*}
This concludes the proof of the Lemma.
\end{proof}
\section{Tightness and the proof of the theorem}
Equation \eqref{eqn_n} can be written in the following way
\begin{eqnarray*}
 u_n (t)&=& u_{0,n}+\lambda_1\int_0^t F^1_n(u_n(s))\, ds-\lambda_2\int_0^t F^2_n(u_n(s))\, ds \\
 &+&\frac{1}{2}\lambda_3^2\int_0^t\left|G_n^2u_n(s)\right|^2ds+ \int_0^tG_nu(s) \,dW(s),\\
&=:& u_{0,n}+\sum_{i=1}^4 u_n^i(t)\quad t\ge 0.
\end{eqnarray*}
Our first aim is to prove the following Lemma.
\begin{lemma}\label{lem-time} There exists  $C>0$ such that for all $n\in\mathbb{N}$
\begin{eqnarray*}
\mathbb{E} \vert u_n^1\vert_{W^{1,2}(0,T;\mathbb{L}^2)}^2 &\leq & C,\\
\mathbb{E} \vert u_n^2\vert_{W^{1,\frac43}(0,T;\mathbb{L}^\frac65)}^{\frac43} &\leq & C,\\
\vert u_n^3\vert_{W^{1,2}(0,T;\mathbb{L}^2)}^2 &\leq & C, \mathbb{P}-\mbox{a.s.}.
\end{eqnarray*}

For all $p\in [2,\infty)$ and  $\alpha \in (0,\frac12) $ there exists  $C>0$ such that for all $n\in\mathbb{N}$
\begin{eqnarray}\label{ineq_u_n^4}
\mathbb{E} \vert u_n^4\vert_{W^{\alpha,p}(0,T;\mathbb{L}^2)}^2 &\leq & C.
\end{eqnarray}
If $\alpha \in (\frac14,\frac12)$ and $q>\frac1{\alpha-\frac14}$, then
\begin{eqnarray}
\label{ineq_u_n}
\sup_{n\in\mathbb{N}} \, \mathbb{E} \vert u_n\vert_{W^{\alpha,q}(0,T;\mathbb{L}^\frac65)}^2 &<&\infty.
\end{eqnarray}
\end{lemma}
\begin{proof}
The first three of the above inequalities follow from Theorem \ref{thm_apriori}.\\
 In order to prove \eqref{ineq_u_n^4} let us first observe that in view of the first a'priori estimates in Theorem \ref{thm_apriori}, for every $p\in [2,\infty)$ one can find $C>0$ such that
$$ \sup_{n\in\mathbb{N}} \, \mathbb{E}\left[ \sup_{t\in [0,T]} \vert  u_n(t)\vert_{\mathbb{L}^2}^p\right] <\infty.
$$
Thus, inequality \eqref{ineq_u_n^4}  is a direct consequence of inequality \eqref{ineq_G_n}    and   Lemma \ref{lem_besov_integral}.

The inequality \eqref{ineq_u_n} follows from the other inequalities since
\begin{trivlist}
\item[(i)] $\mathbb{L}^2 \embed \mathbb{L}^\frac65$, and,
\item[(ii)] by   \cite[Corollary 18, p.138]{Simon_1990},  $W^{1,\frac43}(0,T;E)\subset W^{\alpha,q}(0,T;E)$ if $1-\frac34>\alpha-\frac1q$.
\end{trivlist}
\end{proof}
\begin{lemma}\label{ltight}
The sequence $\left(u_n\right)$ is tight on the space $L^2(0,T;\mathbb{L}^6)\cap W^{\alpha,q}(0,T;\mathbb{H}^{-1,2})$.
\end{lemma}
\begin{proof}
We will apply Lemma \ref{lem_compact_1} with the following choice of Banach spaces: $B_0=\mathbb{H}^{1,2}$, $B_1=\mathbb{L}^{6}$ and will apply Lemma \ref{lem_compact_2} with the following choice of Banach spaces \footnote{Let us remark, that the space $\mathbb{H}^{-1,2}$ plays purely auxiliary r\^ole and it can be replaced by Sobolev spaces of higher (but still negative) regularity order, e.g. $H^{-\delta,\frac65}$. However, we have decided on the former as it is a Hilbert space.}: $X_0=\mathbb{L}^{\frac65}$, $X=H^{-1,2}$. Since the Gagliardo-Nirenberg inequality yields $\mathbb{H}_0^{1,2}\embed L^6(D)$ densely and continuously, by duality we obtain $ \mathbb{L}^{\frac65}\embed \mathbb{H}^{-1,2}$.
Let us choose $\alpha \in (\frac14,\frac12)$ and $q>\frac1{\alpha-\frac14}$, e.g. $\alpha=\frac38$ and $q=9$. Note that then  $q>\frac1{\alpha}$ so that, see \cite{Simon_1990}, $W^{\alpha,q}(0,T;\mathbb{H}^{-1,2})\embed C([0,T];\mathbb{H}^{-1,2})$.
Let us recall, see \cite{adams},  that the embeddings $\mathbb{H}^{1,2} \embed \mathbb{L}^6$  and $ \mathbb{L}^\frac65)\embed  \mathbb{H}^{-1,2}$\del{, where $\delta>0$, } are  compact. Therefore, by lemmas \ref{lem_compact_1}, \ref{lem_compact_2} and \ref{lem-time}, and Theorem \ref{thm_apriori} the sequence $\left(u_n\right)$ is tight on the space $L^2(0,T;\mathbb{L}^6)\cap W^{\alpha,q}(0,T;\mathbb{H}^{-1,2})$.
\end{proof}
\subsection{Proof of the existence of a solution}
By Lemma \ref{ltight} we can find a subsequence, denoted in the same way as the full sequence, such that the laws $\mathcal{L}\left(u_n\right)$ converge weakly to a certain probability measure $\mu$ on  $L^2(0,T;\mathbb{L}^6)\cap W^{\alpha,q}(0,T;\mathbb{H}^{-1,2})$.
\begin{lemma}\label{skor}
There exists a probability space $\left(\Omega^\prime,\mathcal{F}^\prime,\P^\prime\right)$ endowed with the filtration $\left(\mathcal F_t^\prime\right)$ such that\\
(a) $\mathcal F_0^\prime$ contains all $\P^\prime$-null sets,\\
(b) there exists a sequence $(u_n^\prime)$  of $L^2(0,T;\mathbb{L}^6)\cap W^{\alpha,q}(0,T;\mathbb{H}^{-1,2})$-valued random variables defined on $\left(\Omega^\prime,\mathcal{F}^\prime,\P^\prime\right)$ such that the laws of $u_n$ and $u_n^\prime$ on $L^2(0,T;\mathbb{L}^6)\cap W^{\alpha,q}(0,T;\mathbb{H}^{-1,2})$ are equal,\\
(c) there exists an $L^2(0,T;\mathbb{L}^6)\cap W^{\alpha,q}(0,T;\mathbb{H}^{-1,2})$-valued random variable $u^\prime$ defined on $\left(\Omega^\prime,\mathcal{F}^\prime,\P^\prime\right)$ such that $\mathcal{L}\la u^\prime\ra$ on $L^2(0,T;\mathbb{L}^6)\cap W^{\alpha,q}(0,T;\mathbb{H}^{-1,2})$ is equal to $\mu$ and
\begin{equation}
\label{eg_conv_strong}
u_n^\prime\to u^\prime \mbox{ in } L^2(0,T;\mathbb{L}^6)\cap W^{\alpha,q}(0,T;\mathbb{H}^{-1,2}),\; \mathbb{P}^\prime-\mbox{almost surely}.
\end{equation}
(d) there exists an $\left(\mathcal F_t^\prime\right)$-adapted real Brownian Motion $W^\prime$ independent of $u^\prime$.
\end{lemma}
\begin{proof}
The existence of a probability space on which (a), (b) and (c) hold follows in a standard way from the Skorohod embedding Theorem, see \cite{Viot_1976,Metivier_1988,Kozlov_1978,Gatarek_Goldys_1994,Flandoli_G_1995}. Part (d) is obtained by augmenting the probability space given by the Skorohod Theorem.
\end{proof}
By Lemma \ref{skor} we may assume that
\begin{equation}\label{eqn_con_L^2}
\mathbb{E}^\prime \int_0^T \vert u_n^\prime(t)-u^\prime(t)|^2_{\mathbb{L}^6}\, dt \to 0.
\end{equation}
Moreover, the sequence  $(u_n^\prime)$ satisfies the same estimates as the original sequence $(u_n)$. In particular, estimates from Lemma \ref{lem-time} hold, i.e.
for $\alpha \in (\frac14,\frac12)$ and $q>\frac1{\alpha-\frac14}$,
\begin{eqnarray}
\label{ineq_u_n^prime}
\sup_{n\in\mathbb{N}} \, \mathbb{E}^\prime \vert u_n^\prime\vert_{W^{\alpha,q}(0,T;\mathbb{L}^\frac65)}^2 &<&\infty.
\end{eqnarray}
and from Theorem \ref{thm_apriori}
\begin{eqnarray}
\label{ineq_ap_1^prime}
\sup_{t\in [0,T]} \vert u_n^\prime(t)\vert_{\mathbb{L}^2} &\leq &C, \;\mathbb{P}-\mbox{a.s}\\
\label{ineq_ap_2^prime}
 \sup_{n\in\mathbb{N}} \, \mathbb{E}^\prime\left[ \sup_{t\in [0,T]} \vert \nabla u_n^\prime(t)\vert_{\mathbb{L}^2}^2\right] &<&\infty,\\
\label{ineq_ap_3^prime}
\sup_{n\in\mathbb{N}} \, \mathbb{E}^\prime \int_0^T \vert  u_n^\prime(t) \times \Delta u_n^\prime(t) \vert_{\mathbb{L}^2}^2\, dt &<&\infty,\\
\label{ineq_ap_4^prime}
\sup_{n\in\mathbb{N}} \, \mathbb{E}^\prime \int_0^T \vert u_n^\prime(t) \times \big( u_n^\prime(t) \times \Delta u_n^\prime(t)\big) \vert_{\mathbb{L}^{6/5}}^{\frac43}\, dt &<&\infty.
\end{eqnarray}
By a standard subsequence argument, the above inequalities imply that the limiting process $u^\prime$ defined in Lemma \ref{skor} enjoys the following property:
\begin{eqnarray}
\label{ineq_ap_2^full}
 \mathbb{E}^\prime \sup_{t\in [0,T]} \left[ \vert u^\prime(t)\vert_{\mathbb{L}^2}^2 + \vert \nabla u^\prime(t)\vert_{\mathbb{L}^2}^2\right] &<&\infty.
\end{eqnarray}
We can also assume that $u_n^\prime \times \Delta u_n^\prime $ converges weakly in $L^2(\Omega^\prime;L^2(0,T;\mathbb{L}^2))$ to some process $Y$ and
$u_n^\prime \times \big( u_n^\prime \times \Delta u_n^\prime\big)$ converges weakly in $L^{\frac43}(\Omega^\prime;L^{\frac43}(0,T;\mathbb{L}^{\frac65}))$ to some process $Z$, such that
\begin{eqnarray}
\label{ineq_ap_3^full}
 \mathbb{E}^\prime \int_0^T \vert  Y(t) \vert_{\mathbb{L}^2}^2\, dt &<&\infty,\\
\label{ineq_ap_4^full}
 \mathbb{E}^\prime \int_0^T \vert Z(t) \vert_{\mathbb{L}^{6/5}}^{\frac43}\, dt &<&\infty.
\end{eqnarray}
Inequality \eqref{ineq_ap_2^full} implies that $u^\prime \in L^\infty(0,T;\mathbb{H}^{1,2})$ a.s. Moreover, we can assume that for any $q<\infty$
\begin{equation}\label{eqn_con_weak_gradient_norm}
u_n^\prime \to u^\prime  \;\; \mbox{ weakly in } L^2(\Omega^\prime;L^q(0,T;\mathbb{H}^{1,2})).\end{equation}

Using the definition of the distribution $u\times\Delta u$ introduced below Definition \ref{def_sol_mart} we formulate the following
\begin{lemma}\label{lem_proof_1} Assume that $\vp$ is a progressively measurable process such that
\begin{trivlist}
\item[(a)]
$\nabla \vp $ belongs to $L^2(\Omega^\prime;L^2(0,T);\mathbb{L}^3)$ and
 \item[(b)] $u^\prime \times \frac{\partial \vp}{\partial x_i}$ belong to $L^2(\Omega^\prime;L^2(0,T);\mathbb{L}^2)$.
 \end{trivlist}
  Then
\begin{equation}\label{eqn_Y=u times u}
\lim_{n\to\infty}\mathbb{E}^\prime \int_0^T \lb u_n^\prime\times \Delta u_n^\prime,\vp\rb_{\mathbb{L}^2}\, dt= \mathbb{E}^\prime \int_0^T \lb Y(s),\vp\rb_{\mathbb{L}^2}\, dt.
\end{equation}
In particular, $Y$ is equal to $u^\prime \times \Delta u^\prime $.
\end{lemma}
\begin{proof} Note that any   test function $\vp$, i.e. a $C^2$ function $\vp:\bar{D}\to\mathbb{R}^3$ of compact support, satisfies the assumptions of the Lemma. Hence  we only need to prove the first part of the Lemma.
 By Corollary \ref{cor_Green}  we have
\begin{eqnarray}\label{eqn_Green3}
-\lb u_n^\prime\times \Delta u_n^\prime,\vp\rb&=& \sum_{i} \lb  \frac{\partial  u_n^\prime}{\partial x_i} , u_n^\prime\times \frac{\partial  \vp}{\partial x_i}\rb.
\end{eqnarray}
Indeed, the call of $u_n^\prime$ is supported by $C([0,T];\rH_n)$ and $\rH_n\subset D(\rA)$.
For each $i\in\{1,2,3\}$ we have
\begin{eqnarray}\label{eqn_difference}
\lb  \frac{\partial  u_n^\prime}{\partial x_i} , u_n^\prime\times \frac{\partial  \vp}{\partial x_i}\rb
-\lb  \frac{\partial  u^\prime}{\partial x_i} , u^\prime\times \frac{\partial  \vp}{\partial x_i}\rb
&=& \lb \frac{\partial  u_n^\prime}{\partial x_i}- \frac{\partial  u^\prime}{\partial x_i} , u^\prime\times \frac{\partial  \vp}{\partial x_i}\rb
\\&+& \lb  \frac{\partial  u_n^\prime}{\partial x_i} , \big( u_n^\prime-u^\prime \big)\times \frac{\partial  \vp}{\partial x_i}\rb .
\nonumber
\end{eqnarray}
Because $\nabla \vp \in L^2(\Omega^\prime;L^2(0,T);\mathbb{L}^3)$,  in view of \eqref{ineq_holder}, \eqref{eqn_con_L^2} and \eqref{ineq_ap_2^prime}, we infer that
\begin{eqnarray}\label{eqn_con_L^2_b}
  \mathbb{E}^\prime  \int_0^T \vert  \lb  \frac{\partial  u_n^\prime(t)}{\partial x_i} , \big( u_n^\prime(t)-u^\prime(t) \big)\times \frac{\partial  \vp(t)}{\partial x_i}\rb \vert \, dt  &\leq&
 \mathbb{E}^\prime \sup_{t\in [0,T]} \vert    \frac{\partial  u_n^\prime(t)}{\partial x_i}\vert_{\mathbb{L}^2}^2 \\
 &&\hspace{-5truecm}\lefteqn{
 \times  \mathbb{E}^\prime\int_0^T \vert u_n^\prime(t)-u^\prime(t)\vert_{\mathbb{L}^6}^2 \,dt \times \mathbb{E}^\prime \int_0^T \vert \frac{\partial  \vp(t)}{\partial x_i}\vert_{\mathbb{L}^3}^2\, dt \to 0.}
 \nonumber
\end{eqnarray}
Since by the assumptions $u^\prime \times \frac{\partial \vp}{\partial x_i}\in L^2(\Omega^\prime;L^2(0,T);\mathbb{L}^2)$, by applying \eqref{eqn_con_weak_gradient_norm}, we infer that
\begin{equation}\label{eqn_con_L^2_c}
\lim_{n\to\infty}\mathbb{E}^\prime \int_0^T \lb \frac{\partial u_n^\prime}{\partial x_i}-\frac{\partial u^\prime}{\partial x_i}, u^\prime(t)\times \frac{\partial \vp}{\partial x_i}\rb_{\mathbb{L}^2} \, dt= 0.
\end{equation}
Hence we have proved that
\begin{equation}\label{eqn_con_L^2_d}
\lim_{n\to\infty}\mathbb{E}^\prime \int_0^T \lb \frac{\partial u_n^\prime}{\partial x_i}, u_n^\prime(t)\times \frac{\partial \vp}{\partial x_i}\rb_{\mathbb{L}^2} \, dt= \mathbb{E}^\prime \int_0^T \lb \frac{\partial u}{\partial x_i}, u^\prime(t)\times \frac{\partial \vp}{\partial x_i}\rb_{\mathbb{L}^2} \, dt.
\end{equation}
This,  in conjunction with \eqref{eqn_Green3} and \eqref{eqn_difference}, implies that
\begin{equation}\label{eqn_con_L^2_week}
\lim_{n\to\infty}\mathbb{E}^\prime \int_0^T \lb u_n^\prime\times \Delta u_n^\prime,\vp\rb_{\mathbb{L}^2}\, dt= \mathbb{E}^\prime \int_0^T \lb u\times \Delta u,\vp\rb_{\mathbb{L}^2} \,dt.
\end{equation}
what,  by the definition of the process $Y$, implies identity \eqref{eqn_Y=u times u}.
This  concludes the proof of the Lemma.
\end{proof}
\begin{lemma}\label{lem_proof_2} Let $\psi:\bar{D}\to\mathbb{R}^3$ be $C^2$ function with compact support in $D$. Then
\begin{equation}\label{eqn_conv_Z}
\lim_{n\to\infty}\mathbb{E}^\prime \int_0^T \lb u_n^\prime\times \big( u_n^\prime\times \Delta u_n^\prime\big) ,\psi\rb\, dt= \mathbb{E}^\prime \int_0^T \lb Z(s),\psi\rb\, dt.
\end{equation}
In particular, $Z$ is equal to $u^\prime \times \big( u^\prime \times \Delta u^\prime\big)$.
\end{lemma}
\begin{proof} Let $\psi$ be a $C^2$ function $\psi:\bar{D}\to\mathbb{R}^3$ of compact support. Put $Y_n=u_n^\prime \times \Delta u_n^\prime$ and $Y=u\times \Delta u$.  Then
\begin{eqnarray}
\lb u_n \times Y_n,\psi\rb- \lb u \times Y,\psi\rb &=&  \lb  Y_n, u_n \times \psi\rb- \lb  Y, u \times  \psi\rb \\
&=& \lb  Y_n-Y, u \times \psi\rb+ \lb  Y_n, (u_n-u) \times  \psi\rb.
\end{eqnarray}
Since by \eqref{ineq_ap_2^full} the process $\vp:= u \times \psi$ satisfies the assumptions (a) and (b) of the previous Lemma we infer that
$\mathbb{E}^\prime \int_0^T \lb  Y_n(t)-Y(t), u(t) \times \psi(t)\rb \, dt$ converges to $0$ as $n\to\infty$.

Since by \eqref{ineq_ap_3^prime} the sequence $(Y_n)_{n\in\mathbb{N}}$ is bounded in $L^2(\Omega^\prime,L^2(0,T;\mathbb{L}^2))$, by applying \eqref{eqn_con_L^2}  we infer that  $\mathbb{E}^\prime \int_0^T  \lb  Y_n(t), (u_n(t)-u(t)) \times  \psi\rb \, dt $ also converges to $0$ as $n\to\infty$.

Since $\psi \in \left[L^\frac43(\Omega^\prime;L^\frac43(0,T;\mathbb{L}^\frac65))\right]^\ast\cong L^4(\Omega^\prime;L^4(0,T;\mathbb{L}^6))$, by the definition of $Z$ we infer that \eqref{eqn_conv_Z} holds true.
This  concludes the proof of the Lemma.
\end{proof}
So far we have constructed a process $u^\prime$ which may be our  solution. However, a weak martingale solution consists also of a Wiener process and we are going to construct now the latter.

We define a sequence  $M_n=\big(M_n(t)\big)_{t\geq 0}$ of $\rH$-valued processes,  on the probability space $\left(\Omega^\prime,\mathcal F^\prime ,\mathbb P^\prime\right)$ by
\begin{eqnarray}\label{eqn_mart-M_n}
M_n(t)&:=& u_n^\prime(t)-u_n^\prime(0)-\int_0^t F_n(u_n^\prime(s))\, ds -\int_0^t\frac{\lambda_3^2}{2} G_n^2u_n(s))\, ds
\nonumber\\
&=& u_n^\prime(t)-u_n^\prime(0)-\lambda_1 \int_0^t \rpi_n(u_n^\prime(s)\times \Delta u_n^\prime(s) )\, ds  \\
&+&\lambda_2 \int_0^t \rpi_n\big(u_n^\prime(s)\times (u_n^\prime(s)\times \Delta u_n^\prime(s) )\big)\, ds -\frac{\lambda_3^2}{2} \int_0^t \rpi_n[(\rpi_n(u\times h))\times h]\,  ds.
\nonumber
\end{eqnarray}
The next lemma is standard, see for example the corresponding part of the proof of Theorem 8.1 in \cite{DaPrato_Z_1992} on p. 232.
\begin{lemma}\label{lem-mart-M_n}
$M_n$ is a continuous, $\rH_n$-valued,  martingale and  its  quadratic variation  is given by the process
\begin{eqnarray}\label{eqn_var_M_n}
Q_n(t)=\frac12 \lambda_3^2 \int_0^t\rpi_n( u_n^\prime(s)\times h) \otimes \rpi_n( u_n^\prime(s)\times h) \, ds,
\end{eqnarray}
where for $g\in \rH_n$, $g\otimes g : \rH_n \ni x\mapsto \lb x,g\rb g\in \rH_n$. Moreover,
\[\sup_{t\le T}\E\left|M_n(t)\right|^2=\E Q_n(T).\]
\end{lemma}
\begin{remark} Suppose that $\rK$ is a Hilbert space and $X$ is a Banach space. If $Q\in L(\rK,X)$, then by identifying $\rK$ with its dual, $Q\circ Q^\ast\in L(X^\ast,X)$. In fact, the operator $Q Q^\ast=Q\circ Q^\ast$ can be defined by
$$ \lb Q Q^\ast x^\ast,y^\ast\rb= \lb  Q^\ast x^\ast, Q y^\ast\rb_{\rK},\; x^\ast,y^\ast\in X^\ast.
$$
If $\rK$ is one dimensional, any operator $Q\in L(\rK,X)$ cab be identified by a vector $b=b(Q)\in X$. Then one can easily show that $Q Q^\ast=b\otimes b$.
\end{remark}
Let
$$ Q(t):= \frac12 \lambda_3^2 \int_0^t( u^\prime(s)\times h) \otimes ( u^\prime(s)\times h) \, ds,\quad t\geq 0.$$
Clearly, the process $Q$ is a progressively measurable process of trace class operators on $\rH$.
\begin{lemma}\label{lem_conv_M_n}
The process
\begin{equation}\label{eqn_mart-M}
\begin{aligned}
M(t)&:=u^\prime(t)-u^\prime(0)-\lambda_1 \int_0^t u^\prime(s)\times \Delta u^\prime(s))\, ds\\
&+\lambda_2 \int_0^t \big(u^\prime(s)\times (u^\prime(s)\times \Delta u^\prime(s) )\big)\, ds
-\frac{\lambda_3^2}{2} \int_0^t ((u^\prime\times h))\times h)\,  ds ,\,\, t\geq 0.
\end{aligned}
\end{equation}
is a continuous $\mathbb L^2$-valued martingale with the quadratic variation process given by $Q$. Moreover,
\[\E\sup_{t\le T}\left|M(t)\right|^2_{\rH}<\infty\]
\end{lemma}
\begin{proof} Note first that for every $t\ge 0$
\begin{equation}\label{Q_nQ}
\lim_{n\to\infty}\lb Q_n(t)g,g\rb=\lb Q(t)g,g\rb,\quad g\in\rH,\quad \mathbb P\,\, a.s.
\end{equation}
Indeed, for any $g\in\rH$ we have
\[\lb Q_n(t)g,g\rb=\frac{1}{2}\lambda_3^2\int_0^t\lb \pi_n\left(u_n^\prime\times h\right),g\rb^2\]
and in view of Theorem \ref{thm_apriori} \eqref{Q_nQ} follows by Dominated Convergence.\\
By Theorem \ref{thm_apriori} we have
\[\sup_n\E\sup_{t\le T}\left|M_n(t)\right|^2_{\rH}<\infty\]
and since $M_n$ is a martingale for every $n\ge 1$ it is enough to show that
\[\lim_{n\to\infty}\lb M_n(t),g\rb=\lb M(t),g\rb\quad\mathbb P-a.s.,\quad g\in\rH .\]
In fact, the Burkholder-Gundy-Davis inequality yields
\[\lim_{n\to\infty}\E\sup_{t\le T}\left|\left\langle M_n(t)-M(t),g\right\rangle\right|^2\le C\lim_{n\to\infty}\E\int_0^T\lb \left(Q_n(t)-Q(t)\right)g,g\rb\]
and the claim follows from Theorem \ref{thm_apriori}.
\end{proof}
\subsection{End of proof of Theorem \ref{main}}
\begin{proof}[]
By Lemmas \ref{skor} and \ref{lem_conv_M_n} the probability space $\la\Omega^\prime ,\mathcal F^\prime ,\la\mathcal F_t^\prime\ra ,\mathbb P^prime\ra$ and $\rH$-valued continuous martingale $M$ satisfy all the assumptions of the Martingale Representation Theorem in the version proved in \cite{ondrejat} as Theorem 2. Therefore, there exists on $\Omega^\prime$ a one-dimensional Wiener process $W$ such that
\[M(t)=\int_0^t Gu(s)dW(s),\quad t\ge 0.\]
Taking into account \eqref{ineq_ap_2^full}, Lemma \ref{lem_proof_1} and Lemma \ref{lem_proof_2}, this fact completes the proof of the existence of a weak martingale solution to equation \eqref{eqn_SLL_1} and of part (a) of the theorem. Moreover, invoking \eqref{ineq_ap_3^full} and \eqref{ineq_ap_4^full} we find that for every $t\ge 0$ the weak martingale solution $u$ satisfies the equation
\begin{equation}\label{65}
u(t)=u_0+\lambda_1\int_0^tu(s)\times\Delta u(s)ds-\lambda_2\int_0^tu(s)\times(u(s)\times\Delta u(s))ds+\lambda_3\int_0^t(u(s)\times h)\circ dW(s)
\end{equation}
where the first two integrals are the Bochner integrals in $L^2\left( 0,T;\mathbb L^2\right)$ and $L^{4/3}\left( 0,T;\mathbb L^{6/5}\right)$  respectively and the stochastic integral is the Stratonovich integral in $\mathbb L^2$. In particular we have
$u(t)\times (u(t)\times\Delta u(t))\in\mathbb H^{-1,2}$ for a.a. $t\ge 0$.\\
We will show \eqref{sphere}. To this end let $\phi\in C_0^\infty(D,\mathbb R)$. Then by Definition \ref{def_sol_mart} and the It\^{o} formula
\[\langle u(t),u(t)\phi\rangle = \langle u_0,u_0\phi\rangle+2\int_0^t\langle u(s),\circ du(s)\rangle=\langle u_0,u_0\phi\rangle ,\]
where $\circ du(s)$ is a Stratonovitch integral.
Since $\varphi$ is arbitrary and $\left|u_0(x)\right|=1$ for a.a. $x\in D$ we infer that $|u(t,x)|=1$ for a.a. $x\in D$ as well. Now part (d) of the theorem and \eqref{eqn_def_SLL_2} imply easily that
\begin{equation}\label{triple}
u\times (u\times\Delta u)\in L^2\left(0,T;\mathbb L^2\right)
\end{equation}
 and therefore (b) of the theorem holds. Finally, invoking \eqref{eqn_def_SLL_2}, \eqref{triple} the Kolmogorov continuity test and \eqref{ineq_ap_2^full} we obtain (c). This completes the proof of the theorem.
\end{proof}
\begin{appendix}
\section{}
 \label{app-A}
In this Appendix we will list all algebraic identities used in this paper.
Assume that $a,b,c,d\in\mathbb{R}^3$. Then
\begin{eqnarray}
\label{eqn_A-2}
\lb a\times (b\times c),d\rb&=&\lb c,(d\times a)\times b\rb,
\\
\label{eqn_A-3}
\lb a\times b,c\rb&=&\lb b,c\times a\rb,
\\
\label{eqn_A-5}
 -\lb a\times b,c\rb&=&\lb b,a\times c\rb,\\
\label{eqn_A-6}
a\times(b\times c)&=&\lb a,c\rb b- \lb a,b\rb c
\\
\label{eqn_A-7}
\vert a\times b\vert &\leq & \vert a\vert \vert b\vert.\\
\label{eqn_A-10}
\lb a \times (a\times b),b\rb&=&-\vert a\times b\vert^2
\end{eqnarray}
In particular, if $\lb a,b\rb=0$, then $(a\times b)\times b= b \times(b\times a)=\lb b,a\rb b-\lb b,b\rb a=-\vert b\vert^2 a$ and
$a\times (a\times b)=\lb a,b\rb a-\lb a,a\rb b=-\vert a\vert^2 b$, i.e.
\begin{eqnarray}
\label{eqn_A-9}
(a\times b) \times b&=&-\vert b\vert^2 a, \, \text{if}\, \lb a,b\rb=0.
\\
\label{eqn_A-9b}
a\times (a\times b)&=& -\vert a\vert^2 b, \, \text{if}\, \lb a,b\rb=0.
\end{eqnarray}
\section{}
 \label{app-B}
Let us formulate the following simple consequence of the H{\"o}lder inequality and the classical inequality \eqref{eqn_A-7}.
\begin{proposition}\label{prop-1} Assume that $\frac1p+\frac1q=\frac1r$. then for all $u\in L^p(D,\mathbb{R}^3)$ and  $u\in L^q(D,\mathbb{R}^3)$,
\begin{eqnarray}\label{ineq_holder}
\vert u\times v\vert_{_{L^r}} &\leq&  \vert u \vert_{_{L^p}} \vert v\vert_{_{L^q}},
\end{eqnarray}
If $r\in [1,\frac32]$, then for all $u\in L^p(D,\mathbb{R}^3)$ and  $u\in L^q(D,\mathbb{R}^3)$,
\begin{eqnarray}\label{ineq_L^6}
\vert u\times v\vert_{_{L^r}} &\leq&  \vert u \vert_{_{L^2}}^{\frac3r-2} \vert u \vert_{_{L^6}}^{3-\frac3r} \vert v\vert_{_{L^2}}.
\end{eqnarray}
\end{proposition}
\begin{proof}[Proof of inequality \eqref{ineq_L^6}]
The following are special cases of \eqref{ineq_holder}:
\begin{eqnarray*}
\vert u\times v\vert_{_{L^1}} \leq  \vert u \vert_{_{L^2}} \vert v\vert_{_{L^2}}, \;&&\; \vert u\times v\vert_{_{L^\frac32}} \leq  \vert u \vert_{_{L^6}} \vert v\vert_{_{L^2}},\
\end{eqnarray*}
Inequality \eqref{ineq_L^6} follows then  from \del{the above two inequalities and} the Riesz Interpolation Theorem, see \cite{Bergh_L_1976}.
\end{proof}
For the reader's convenience we will recall some facts that will be crucial for the proof of tightness of the approximating sequence $\left(u_n\right)$. Let us recall first that the Sobolev space $W^{1,q}(0,T;E)$, where $q\in [1,\infty)$ and $E$ is a separable Banach space is the space of all functions $u\in L^p(0,T;E)$ that are weakly differentiable and its weak derivative $u^\prime$ also belongs to $L^p(0,T;E)$. If $\alpha \in (0,1)$ and $q\in [1,\infty)$, then the Besov-Slobodetski space $W^{\alpha,q}(0,T;E)$  is the space of all $u\in L^q(0,T;E)$ such that
$$\int_0^T \int_0^T \frac{|u(t)-u(s)|^q}{|t-s|^{1+\alpha q}}<\infty.$$
For $\alpha \in(0,1]$, endowed with a norm
\begin{equation}
\label{eqn_besov_norm}
\Vert u\Vert_{W^{\alpha,q}(0,T;E)} =\begin{cases} \left[\int_0^T|u(t)|^q\, dt+ \int_0^T \int_0^T \frac{|u(t)-u(s)|^q}{|t-s|^{1+\alpha q}}\, dt \,ds\right]^{1/q}, &\text{ if } \;\alpha\in (0,1);\\
 \int_0^T|u(t)|^q\, dt+ \int_0^T|u^\prime(t)|^q\, dt, &\text{ if } \;\alpha =1,
\end{cases}
\end{equation}
$W^{\alpha,q}(0,T;E)$   is a separable Banach space. It is known, see e.g. \cite{Simon_1990} for a direct treatment, that
 $W^{\alpha,q}(0,T;E)\embed  W^{\beta,q}(0,T;E)$ if $\beta\leq \alpha\leq 1$ and
 $W^{\alpha,q}(0,T;E)\embed C^\delta([0,T];E) $ continuously provided that $\delta\geq 0$ and  $\alpha>\delta+\frac1q$.

The following result is just  Lemma 2.1 from \cite{Flandoli_G_1995}.

\begin{lemma}
\label{lem_besov_integral}
Assume that $E$ is a separable Hilbert space, $p\in [2,\infty)$  and $\alpha\in (0,\frac12)$.    Then there exists a constant $C$ depending on $T$ and $\alpha $, such that for all $\xi\in \mathcal{M}^p(0,T,E)$,
\begin{eqnarray}\label{ineq_besov_integral}
\mathbb{E} \Vert I(\xi)\Vert^p_{W^{\alpha,p}(0,T;E)} &\leq & C \mathbb{E} \int_0^T  |\xi(r)|^p_E\, dt,\\
&&\hspace{-7truecm}\lefteqn{\mbox{where the process } I(\xi) \mbox{ is defined by }} \nonumber\\
I(\xi)&:= &\int_0^t\,\xi(s)\, dW(s),\; t\geq 0.
\end{eqnarray}
In particular,  $\mathbb{P}$-a.s. the trajectories of the process $I(\xi)$  belong to $W^{\alpha,2}(0,T;E)$.
\end{lemma}
\section{}
 \label{app-C}
 We will need the following  two compactness results. For the first one  see  Theorem 2.1 in \cite{Flandoli_G_1995} which  is a modification of results  in  section I.5 of  \cite{Lions_1969} and section 13.3 of \cite{Temam_1983}. The second one is related to Theorem 2.2 in \cite{Flandoli_G_1995} and will be proven in the Appendix.
\begin{lemma}\label{lem_compact_1}
Assume that $B_0\subset B\subset B_1$ are Banach spaces, $B_0$ and $B_1$ being reflexive. Assume that the embedding $B_0\subset B$ is compact, $q\in (1,\infty)$ and $\alpha \in (0,1)$. Then the embedding
\begin{equation}\label{eqn_compact_1}
L^p(0,T;B_0)\cap W^{\alpha,q}(0,T;B_1) \embed L^p(0,T;B)
\end{equation}
is compact.
\end{lemma}

\begin{lemma}\label{lem_compact_2}
Assume that $X_0\subset X$ are Banach spaces such  that the embedding $X_0\subset X$ is compact. Assume that  $q\in (1,\infty)$ and $0<\alpha<\beta<1 $. Then the embedding $W^{\beta,q}(0,T;X_0)\subset W^{\alpha,q}(0,T;X)$ is compact.
\end{lemma}
\end{appendix}

\def\polhk#1{\setbox0=\hbox{#1}{\ooalign{\hidewidth
    \lower1.5ex\hbox{`}\hidewidth\crcr\unhbox0}}}

\end{document}